\documentclass[12pt]{amsart}
\usepackage{amscd,amsmath,amsthm,amssymb}
\usepackage{pstcol,pst-plot,pst-3d}
\usepackage{color}
\usepackage{pstricks}
\usepackage{stmaryrd}
\usepackage{tikz}
\usepackage{lineno}

\newpsstyle{fatline}{linewidth=1.5pt}
\newpsstyle{fyp}{fillstyle=solid,fillcolor=verylight}
\definecolor{verylight}{gray}{0.97}
\definecolor{light}{gray}{0.9}
\definecolor{medium}{gray}{0.85}
\definecolor{dark}{gray}{0.6}

\def\G{{\mathcal G}}
\def\F{{\mathcal F}}

\def\I{{\mathcal I}}

\def\opn#1#2{\def#1{\operatorname{#2}}} 
\opn\chara{char} \opn\length{\ell} \opn\pd{pd} \opn\rk{rk}
\opn\projdim{proj\,dim} \opn\injdim{inj\,dim} \opn\rank{rank}
\opn\depth{depth} \opn\grade{grade} \opn\height{height}
\opn\embdim{emb\,dim} \opn\codim{codim}

\opn\Tr{Tr} \opn\bigrank{big\,rank}
\opn\superheight{superheight}\opn\lcm{lcm}
\opn\trdeg{tr\,deg}
\opn\reg{reg} \opn\lreg{lreg} \opn\ini{in} \opn\lpd{lpd}
\opn\size{size} \opn\sdepth{sdepth}
\opn\link{link}\opn\fdepth{fdepth}\opn\lex{lex}
\opn\tr{tr}
\opn\type{type}
\opn\Borel{Borel}
\opn\cdeg{cdeg}

\opn\div{div} \opn\Div{Div} \opn\cl{cl} \opn\Cl{Cl}
\opn\Spec{Spec} \opn\Supp{Supp} \opn\supp{supp} \opn\Sing{Sing}
\opn\Ass{Ass} \opn\Min{Min}\opn\Mon{Mon}
\opn\Ann{Ann} \opn\Rad{Rad} \opn\Soc{Soc}
\opn\Im{Im} \opn\Ker{Ker} \opn\Coker{Coker} \opn\Am{Am}
\opn\Hom{Hom} \opn\Tor{Tor} \opn\Ext{Ext} \opn\End{End}
\opn\Aut{Aut} \opn\id{id}

\opn\nat{nat}
\opn\pff{pf}
\opn\Pf{Pf} \opn\GL{GL} \opn\SL{SL} \opn\mod{mod} \opn\ord{ord}
\opn\Gin{Gin} \opn\Hilb{Hilb}\opn\sort{sort}
\opn\PF{PF}\opn\Ap{Ap}
\opn\aff{aff} \opn
\con{conv} \opn\relint{relint} \opn\st{st}
\opn\lk{lk} \opn\cn{cn} \opn\core{core} \opn\vol{vol}  \opn\inp{inp} \opn\nilpot{nilpot}
\opn\link{link} \opn\star{star}\opn\lex{lex}\opn\set{set}
\opn\width{wd}
\opn\Fr{F}
\opn\QF{QF}
\opn\G{G}
\opn\type{type}\opn\res{res}
\opn\gr{gr}

\def\cdeg{deg}

\def\pot#1#2{#1[\kern-0.28ex[#2]\kern-0.28ex]}

\opn\dirlim{\underrightarrow{\lim}}
\opn\inivlim{\underleftarrow{\lim}}
\def\Implies{\ifmmode\Longrightarrow \else
	\unskip${}\Longrightarrow{}$\ignorespaces\fi}
\def\implies{\ifmmode\Rightarrow \else
	\unskip${}\Rightarrow{}$\ignorespaces\fi}
\def\iff{\ifmmode\Longleftrightarrow \else
	\unskip${}\Longleftrightarrow{}$\ignorespaces\fi}

\let\:=\colon
\def\opn#1#2{\def#1{\operatorname{#2}}} 
\opn\chara{char} \opn\length{\ell} \opn\pd{pd} \opn\rk{rk}
\opn\projdim{proj\,dim} \opn\injdim{inj\,dim} \opn\rank{rank}
\opn\depth{depth} \opn\grade{grade} \opn\height{height}
\opn\embdim{emb\,dim} \opn\codim{codim}

\opn\Tr{Tr} \opn\bigrank{big\,rank}
\opn\superheight{superheight}\opn\lcm{lcm}
\opn\trdeg{tr\,deg}
\opn\reg{reg} \opn\lreg{lreg} \opn\ini{in} \opn\lpd{lpd}
\opn\size{size} \opn\sdepth{sdepth}
\opn\link{link}\opn\fdepth{fdepth}\opn\lex{lex}
\opn\tr{tr}
\opn\type{type}
\opn\gap{gap}
\opn\div{div} \opn\Div{Div} \opn\cl{cl} \opn\Cl{Cl}
\opn\Spec{Spec} \opn\Supp{Supp} \opn\supp{supp} \opn\Sing{Sing}
\opn\Ass{Ass} \opn\Min{Min}\opn\Mon{Mon}
\opn\Ann{Ann} \opn\Rad{Rad} \opn\Soc{Soc}
\opn\Im{Im} \opn\Ker{Ker} \opn\Coker{Coker} \opn\Am{Am}
\opn\Hom{Hom} \opn\Tor{Tor} \opn\Ext{Ext} \opn\End{End}
\opn\Aut{Aut} \opn\id{id}

\opn\nat{nat}
\opn\pff{pf}
\opn\Pf{Pf} \opn\GL{GL} \opn\SL{SL} \opn\mod{mod} \opn\ord{ord}
\opn\Gin{Gin} \opn\Hilb{Hilb}\opn\sort{sort}
\opn\PF{PF}\opn\Ap{Ap}
\opn\aff{aff}
\opn\relint{relint} \opn\st{st}
\opn\lk{lk} \opn\cn{cn} \opn\core{core} \opn\vol{vol}  \opn\inp{inp} \opn\nilpot{nilpot}
\opn\link{link} \opn\star{star}\opn\lex{lex}\opn\set{set}
\opn\width{wd}
\opn\Fr{F}
\opn\QF{QF}
\opn\G{G}
\opn\type{type}\opn\res{res}
\opn\conv{conv}
\opn\gr{gr}

\def\pot#1#2{#1[\kern-0.28ex[#2]\kern-0.28ex]}

\opn\dirlim{\underrightarrow{\lim}}
\opn\inivlim{\underleftarrow{\lim}}
\opn\aff{aff} \opn
\con{conv} \opn\relint{relint} \opn\st{st}
\opn\lk{lk} \opn\cn{cn} \opn\core{core} \opn\vol{vol}  \opn\inp{inp} \opn\nilpot{nilpot}
\opn\link{link} \opn\star{star}\opn\lex{lex}\opn\set{set}
\opn\width{wd}
\opn\Fr{F}
\opn\QF{QF}
\opn\G{G}
\opn\type{type}\opn\res{res}
\opn\gr{gr}

\def\cdeg{deg}

\def\pot#1#2{#1[\kern-0.28ex[#2]\kern-0.28ex]}

\opn\dirlim{\underrightarrow{\lim}}
\opn\inivlim{\underleftarrow{\lim}}
\def\Implies{\ifmmode\Longrightarrow \else
	\unskip${}\Longrightarrow{}$\ignorespaces\fi}
\def\implies{\ifmmode\Rightarrow \else
	\unskip${}\Rightarrow{}$\ignorespaces\fi}
\def\iff{\ifmmode\Longleftrightarrow \else
	\unskip${}\Longleftrightarrow{}$\ignorespaces\fi}

\let\:=\colon

\opn\dis{dis}
\def\pnt{{\raise0.5mm\hbox{\large\bf.}}}

\opn\Lex{Lex}

\def\Implies{\ifmmode\Longrightarrow \else
	\unskip${}\Longrightarrow{}$\ignorespaces\fi}
\def\implies{\ifmmode\Rightarrow \else
	\unskip${}\Rightarrow{}$\ignorespaces\fi}
\def\iff{\ifmmode\Longleftrightarrow \else
	\unskip${}\Longleftrightarrow{}$\ignorespaces\fi}

\let\:=\colon
\newtheorem{Theorem}{Theorem}[section]
\newtheorem{Lemma}[Theorem]{Lemma}
\newtheorem{Corollary}[Theorem]{Corollary}

\newtheorem{Remark}[Theorem]{Remark}

\newtheorem{Example}[Theorem]{Example}

\newtheorem{Definition}[Theorem]{Definition}

\let\epsilon\varepsilon
\let\kappa=\varkappa
\def\qed{\ifhmode\textqed\fi
	\ifmmode\ifinner\quad\qedsymbol\else\dispqed\fi\fi}
\def\textqed{\unskip\nobreak\penalty50
	\hskip2em\hbox{}\nobreak\hfil\qedsymbol
	\parfillskip=0pt \finalhyphendemerits=0}
\def\dispqed{\rlap{\qquad\qedsymbol}}

\opn\dis{dis}
\def\pnt{{\raise0.5mm\hbox{\large\bf.}}}

\opn\Lex{Lex}

\begin{document}
	\title {The  facet ideals of  matching complexes of line graphs}
	
	\author {Guangjun Zhu$^{^*}$\!\!\!,  Hong Wang, Yijun Cui}

	\address{Authors address:  School of Mathematical Sciences, Soochow
		University, Suzhou 215006, P.R. China}
	\email{zhuguangjun@suda.edu.cn(Corresponding author:Guangjun Zhu),
		\linebreak[4]651634806@qq.com(Hong Wang), 237546805@qq.com(Yijun Cui).}

\dedicatory{ }

\setcounter{tocdepth}{1}

	\dedicatory{ }
	
	\begin{abstract}
	Let $L_n$ be a line graph with $n$ edges and $\F(L_n)$ the facet ideal  of its matching complex. In this paper, we provide the irreducible decomposition of $\F(L_n)$ and some exact formulas for the  projective dimension and the  regularity of  $\F(L_n)$.
	\end{abstract}
	
	\thanks{* Corresponding author}
	
	
	
	
	\maketitle
	
	\setcounter{tocdepth}{1}

\section*{Introduction}
Graph complexes  have provided an important link between combinatorics and algebra, topology, and geometry (see e.g.,\cite{B,J,W}).
Some well-studied examples are  matching complexes and independence complexes of graphs.
Let $G$ be  a finite simple undirected graph  (i.e., loopless and without multiple
edges)  with the vertex set $V(G)$ and the edge set $E(G)$. A subset $M\subset E(G)$ is a {\it matching} of $G$ if $M$ contains no
adjacent edges.  A collection of all matchings  of $G$  forms a  simplicial complex $\mathcal{M}(G)$, whose vertex set is the set of edges of $G$ and whose $k$-dimensional faces are matchings  of size $(k+1)$,
we call $\mathcal{M}(G)$   the {\it matching complex} of $G$. A subset $W\subset V(G)$ is called an {\it independent} set if no edge of $G$ has
both endpoints in $W$. A collection of all independent subsets  of $G$ also forms a  simplicial complex $\text{Ind}(G)$,
whose vertex set is  $V(G)$ and whose face set are all independent sets of $G$, we call $\text{Ind}(G)$ to be the {\it independent complex} of $G$.

Topological and geometric properties of the matching complex and independence complexes  of  some graphs has been studied by many authors (See \cite{BLVZ,E1,E2,EH,FH,K1,K2,RR,Zi}).
The matching complex  of a complete bipartite graph $K_{m,n}$  was first introduced in the  thesis of  Garst \cite{G} dealing with Tits coset complexes.
He showed that $\mathcal{M}(K_{m,n})$ is
Cohen-Macaulay if and only if $n\geq 2m-1$.
Ziegler \cite{Zi} strengthened this result by showing that  $\mathcal{M}(K_{m,n})$ is shellable if $n\geq 2m-1$. Consequently,  $\mathcal{M}(K_{m,n})$ has the homotopy type of a wedge of $(m-1)$-spheres when $n\geq 2m-1$.
Ehrenborg-Hetyei \cite{EH} and Engstr\"om \cite{E2}  showed that the independence complex $\text{Ind}(F)$ of a  forest $F$
is either contractible or is homotopy equivalent to a $(\gamma(F)-1)$-sphere, and is also $\lfloor\frac{n-1}{2d}-1\rfloor$-connected, where $n$ is the number of vertices, $d$ is  the maximal degree of vertices, $\gamma(F)$ denotes the domination number of $F$,   and $\lfloor\frac{n-1}{2d}-1\rfloor$ is the largest integer $\leq \frac{n-1}{2d}-1$.

Let $\Delta$ be a simplicial complex on the vertex set $V=\{x_i\mid 1\leq i\leq n\}$, $\text{Facets}\,(\Delta)$ denote a collection of  all facets (maximal faces under inclusion)  of $\Delta$. By identifying the vertex $x_i$ with the variable $x_i$ in the polynomial ring $S=k[x_1,\dots, x_n]$ over a field $k$, one can associate $\Delta$  with two  squarefree monomial
ideals of $S$	
$$
\I_{\Delta}=(x_{i_1}\cdots x_{i_s}\mid\{x_{i_1},\ldots, x_{i_s}\} \text{\ is not a face of \ } \Delta ),
$$
and
$$
\F(\Delta)=(x_{\ell_1}\cdots x_{\ell_s}\mid\{x_{\ell_1},\ldots, x_{\ell_s}\} \in \text{Facets}\,(\Delta) ).
$$
They are called the  {\it Stanley-Reisner} ideal and the {\it facet ideal} of $\Delta$ respectively.
The Stanley-Reisner ring of $\Delta$ is $k[\Delta]=S/\I_{\Delta}$.

	Some algebraic properties of the matching complexes $\mathcal{M}(K_{m,n})$ have been studied in \cite{BLVZ,FH,JZWZ}.
	Bj\" orner et al. in \cite{BLVZ} showed that  $\depth\,( k[\mathcal{M}(K_{m,n})])=\min\{m,n,\lfloor\frac{m+n+1}{3}\rfloor\}$ for any $m,n$.
	Freidman-Hanlon in \cite{FH}  showed that $b_{r-1}(\mathcal{M}(K_{m,n}))\\
	=0$ if and only if $(m-r)(n-r)>r$, and $b_{v-1}(\mathcal{M}(K_{m,n}))>0$ if and only if $n\geq 2m-4$ or $(m,n)=(6,6),(7,7),(8,9)$,
	where $b_{i}(\mathcal{M}(K_{m,n}))$ is the $i$-th Betti number of $\mathcal{M}(K_{m,n})$, it equals the rank of the homology group $H_i(\mathcal{M}(K_{m,n}))$.
	Jiang et al. in \cite{JZWZ} described the irreducible decomposition of the facet ideal $\F(\mathcal{M}(K_{m,n}))$ of the matching complexe $\mathcal{M}(K_{m,n})$ and provided some lower bounds for depth and regularity of the facet ideal  $\F(\Delta_{m,n})$, where $n\geq m$. They also showed that these lower bounds can be obtained if $m\leq3$.
	
	In this article, we are interested in  algebraic properties of the facet ideal $\F(L_n)$ of the matching complex $\mathcal{M}(L_n)$ of a line graph $L_n$ with $n$ edges.
	By the definition of the matching complex and the independence complex of a line graph,
	we can easily obtain that $\mathcal{M}(L_n)$ is isomorphic to $\text{Ind}\,(L_{n-1})$, and the Stanley-Reisner ideal $\I_{\mathcal{M}(L_n)}$ of $\mathcal{M}(L_n)$
	is actually the edge ideal of a line graph with $n$ vertices, which was studied in \cite{AF,HT,Zhu}.
	This complex has an extremely simple structure, but it is still worth discussing, as they
	appear naturally in many situations. For some examples, see \cite{J,K2}.  The first author in \cite[Theorem 3.3]{Zhu} provided some exact formulas for the  projective dimension and the  regularity of  $\I_{\mathcal{M}(L_n)}$ for any $n\geq 2$,
	she showed that
	\[
	\pd\,(\I_{\mathcal{M}(L_n)})=\left\{\begin{array}{ll}
		2p-1\ \ &\text{if}\ n=3p \text{\ or\ } n=3p+1,\\
		2p\ \ &\text{if}\ n=3p+2,
	\end{array}\right.
	\]
	and
	\[
	\reg\,(\I_{\mathcal{M}(L_n)})=\left\{\begin{array}{ll}
		p+1\ \ &\text{if}\ \ n=3p, \text{\ or\ } n=3p+1,\\
		p+2\ \ &\text{if}\ \ n=3p+2.
	\end{array}\right.
	\]
	See also \cite[Corollary 4.15]{AF} and \cite[Theorem 4.1]{HT}.
	In this paper, we will provide the irreducible decomposition  and some exact formulas for the  projective dimension and the  regularity of  the facet ideal $\F(L_n)$ (see Theorem \ref{decompose2}, Theorem \ref{line2} and Corollary \ref{line3}).

	Our paper is organized as follows.  In the preliminary  section, we collect necessary terminology and results from the
	literature. In Section $2$, we give the irreducible decomposition of the facet ideal $\F(L_n)$ of a line graph $L_n$ with $n$ edges. In Section $3$,
	we give some exact formulas for  projective dimension and the regularity of    $\F(L_n)$ by  some suitable short exact sequences.

	\medskip

	\section{Preliminaries }
	
	In this section, we gather together the needed  notations and basic facts, which will
	be used throughout this paper. However, for more details, we refer the reader to \cite{BH,F,HH1,J}.
	
	Let $n$ be a positive integer and $[n]=\{i\mid 1\leq i\leq n\}$. In this paper, we will assume that $L_n$ is a line graph  with  edge set $\{x_i\mid i\in [n]\}$. For  simplicity,
	we denote by $\F(L_n)$ the facet ideal of the matching complex $\mathcal{M}(L_n)$ of $L_n$.
	\begin{Example} \label{example1} Let $L_6$ be a line graph  with  six  edges  $x_1,\ldots,x_6$. Then the  facet ideal  of its matching complex is
		$$
		\F(L_6)=(x_1x_3x_5,x_1x_3x_6,x_1x_4x_6,x_2x_4x_6,x_2x_5).
		$$
	\end{Example}
	
	A monomial ideal is called {\em irreducible} if it cannot be written as proper intersection of two other monomial ideals.
	It is called {\em reducible} if it is not irreducible. It is well known that a monomial ideal is irreducible if and only if it is generated
	by pure powers of the variables, that is, it  has  the form $(x_{i_1}^{a_1},\ldots,x_{i_k}^{a_k})$. The  following lemma is a fundamental fact.
	\begin{Lemma} \label{decomposition}{\em(\cite[Theorem 1.3.1]{HH1})}
		Let $I\subset S$ be a monomial ideal. Then there exists a unique decomposition
		$$I= Q_1\cap\cdots\cap Q_r$$
		such that  none of the  $Q_i$  can be omitted in this intersection and each  $Q_i$ is an irreducible monomial ideal.
		In particular, if $I$ is  squarefree, then each $Q_i$ is a  minimal {\em(}under inclusion{\em)} prime  ideal over $I$.
	\end{Lemma}
	This decomposition is called  {\it irredundant presentation}  of $I$ and each $Q_i$ is called an {\it irreducible
		component} of $I$.

	\medskip
	\begin{Definition} \label{cover}
		Let $\Delta$  be a simplicial complex on the vertex set $V$. A {\it  vertex cover} of $\Delta$  is a subset $C\subset V$ such that  each facet of $\Delta$ has at least one
		vertex in $C$.  Such a vertex cover $C$ is called minimal if no subset $C'\subsetneq C$ is a vertex cover
		of $\Delta$.
	\end{Definition}
	
	We need the following lemma.
	\begin{Lemma} \label{prime}{\em (\cite[Proposition 1.8]{F})} Let  $\Delta$  be a simplicial complex, $\F(\Delta)$ its facet ideal.
		Then $P=(x_{i_1},\ldots,x_{i_s})$ is a minimal prime ideal over $\F(\Delta)$ if and only if $\{x_{i_j}\mid 1\leq j\leq s\}$ is a minimal vertex cover for $\Delta$.
	\end{Lemma}

	\medskip
	Let  $I\subset S$ be a non-zero homogeneous ideal and
	$$0\rightarrow \bigoplus\limits_{j}S(-j)^{\beta_{p,j}(I)}\rightarrow \bigoplus\limits_{j}S(-j)^{\beta_{p-1,j}(I)}\rightarrow \cdots\rightarrow \bigoplus\limits_{j}S(-j)^{\beta_{0,j}(I)}\rightarrow I\rightarrow 0$$
	is a {\em minimal graded free resolution} of $I$, where $p\leq n$ and $S(-j)$ is an $S$-module obtained by shifting
	the degrees of $S$ by $j$. The number
	$\beta_{i,j}(I)$, the $(i,j)$-th graded Betti number of $I$, is
	an invariant of $I$ that equals the minimal number of  generators of degree $j$ in the
	$i$th syzygy module of $I$.
	Of particular interest is the following invariants which measure the ¡°size¡± of the minimal graded
	free resolution of $I$.
	The {\it projective dimension} of $I$, denoted by $\pd\,(I)$, is defined to be
	$$\mbox{pd}\,(I):=\mbox{max}\,\{i\ |\ \beta_{i,j}(I)\neq 0\}.$$
	The {\it regularity} of $I$, denoted by $\mbox{reg}\,(I)$, is defined by
	$$\mbox{reg}\,(I):=\mbox{max}\,\{j-i\ |\ \beta_{i,j}(I)\neq 0\}.$$

	\medskip
	
	In order to compute the projective dimension and regularity of a non-zero homogeneous ideal,  we shall use the following lemmas in this paper.
	
	\begin{Lemma}\label{sum}{\em (\cite[Lemmas 2.2 and 3.2]{HT1})}
		Let $S_{1}=k[x_{1},\dots,x_{m}]$, $S_{2}=k[x_{m+1},\dots,x_{n}]$ be two polynomial rings and $S=S_1\otimes_k S_2$, let $I\subset S_{1}$,
		$J\subset S_{2}$ be two non-zero homogeneous ideals. Then
		\begin{itemize}
			\item[(1)] $\pd\,(I+J)=\pd\,(I)+\pd\,(J)+1$,
			\item[(2)] $\reg\,(I+J)=\reg\,(I)+\reg\,(J)-1$.
		\end{itemize}
	\end{Lemma}

	\begin{Lemma}\label{exact} {\em \cite[Lemma 1.1 and Lemma 1.2]{HTT}} Let\ \ $0\rightarrow A \rightarrow  B \rightarrow  C \rightarrow 0$\ \  be a short exact sequence of finitely generated graded $S$-modules. Then
		\begin{itemize}
			\item[(1)]$\reg\,(B)\leq \max\,\{\reg\,(A), \reg\,(C)\}$,  the equality holds if $\reg\,(A)-1\neq\reg\,(C)$,
			\item[(2)]$\reg\,(C)\leq \max\,\{\reg\,(A)-1, \reg\,(B)\}$, the equality holds if $\reg\,(A)\neq\reg\,(B)$,
			\item[(3)]$\pd\,(B)\leq \max\,\{\pd\,(A), \pd\,(C)\}$,  the equality holds if $\pd\,(A)+1\neq\pd\,(C)$,
			\item[(4)]$\pd\,(C)\leq \max\,\{\pd\,(A)+1, \pd\,(B)\}$, the equality holds if $\pd\,(A)\neq\pd\,(B)$.
		\end{itemize}
	\end{Lemma}
	
	\medskip
	Let $\mathcal{G}(I)$ denote the minimal set of generators of a monomial ideal $I\subset S$ and let $u\in S$ be a monomial, we set $\mbox{supp}(u)=\{x_i: x_i|u\}$. If $\mathcal{G}(I)=\{u_1,\ldots,u_m\}$, we set $\mbox{supp}(I)=\bigcup\limits_{i=1}^{m}\mbox{supp}(u_i)$. The following lemma is well known.
	\begin{Lemma}\label{supp}
		Let  $I, J=(u)$ be two monomial ideals  such that $\mbox{supp}(u)\cap \mbox{supp}(I)=\emptyset$, where $u$  is a monomial of degree $m$. Then
		\begin{itemize}
			\item[(1)] $\mbox{reg}\,(J)=m$,
			\item[(2)]$\mbox{pd}\,(uI)=\mbox{pd}\,(I)$,
			\item[(3)]$\mbox{reg}\,(uI)=\mbox{reg}\,(I)+m$.
		\end{itemize}
	\end{Lemma}
	
	\section{The irreducible decompositions of the facet ideals of  the matching complexes of  line graphs }
	
	In this section, we will give the irreducible decomposition of the facet ideal $\F(L_{n})$ of  the matching complex of a line graph $L_n$ with $n$ edges.
	
	\medskip
	
	Let $k,n$ be two positive integers with $n>k$, $L'_{n-k}$ denote the line graph with  edge set  $\{x_i\mid i\in [n]\setminus [k]\}$.
	
	\medskip
	\begin{Remark}\label{remark1}
		Let $n$ be a positive integer, $L_n$ a line graph with edge set $\{x_i\mid i\in [n]\}$. Let $\mathcal{M}(L_n)$ be the matching complex of $L_n$. Then $F\in \text{Facets}\,(\mathcal{M}(L_n))$ if and only if $F$ satisfies the following conditions:
		\begin{itemize}
			\item[(1)] If $n=1$, then $F=\{x_1\}$;
			\item[(2)] If $n=2$, then $F=\{x_1\}$ or $F=\{x_2\}$;
			\item[(3)] If $n=3$, then $F=\{x_1,x_3\}$ or $F=\{x_2\}$;
			\item[(4)] If $n\geq 4$, then $F=\{x_{i_j}\mid j\in [k]\}$, where $i_1\in \{1,2\}$, $i_k\in \{n-1,n\}$ and
			$i_{j}-i_{j-1}\in \{2,3\}$ for any $2\leq j\leq k$.
		\end{itemize}
		In particular, if $n\geq 4$, then  $F\in \text{Facets}\,(\mathcal{M}(L_n))$ if and only if $F\setminus \{x_1\}\in \text{Facets}\,(\mathcal{M}(L'_{n-2}))$  when $x_1\in F$, or $F\setminus \{x_2\}\in \text{Facets}\,(\mathcal{M}(L'_{n-3}))$   when $x_2\in F$.
		Similarly, if $n\geq 4$, then $F\in \text{Facets}\,(\mathcal{M}(L_n))$ if and only if $F\setminus \{x_{n-1}\}\in \text{Facets}\,(\mathcal{M}(L_{n-3}))$  when $x_{n-1}\in F$, or $F\setminus \{x_n\}\in \text{Facets}\,(\mathcal{M}(L_{n-2}))$ when $x_{n}\in F$.
	\end{Remark}
	
	\medskip
	From the remark above, we have
	\begin{Remark}\label{remark2}
		Let $n\geq 2$ be an integer, $L_n$ a line graph as in  Remark \ref{remark1}. Let $\F(L_n)$ be the facet ideal
		of the  simplicial complex $\mathcal{M}(L_n)$. Then
		$$\mathcal{G}(\F(L_n))=\mathcal{G}(x_1\F(L'_{n-2}))\cup \mathcal{G}(x_2\F(L'_{n-3}))=\mathcal{G}(x_n\F(L_{n-2}))\cup\mathcal{G}(x_{n-1}\F(L_{n-3})),$$
		where $\{x_1\}\cap \text{supp}\,(\F(L'_{n-3}))=\{x_2\}\cap \text{supp}\,(\F(L'_{n-2}))=\{x_n\}\cap \text{supp}\,(\F(L_{n-3}))=\{x_{n-1}\}\cap \text{supp}\,(\F(L_{n-2}))=\emptyset$.
		For the convenience of marking, we stipulate  $\F(L'_{-1})=\F(L'_{0})=\F(L_{-1})=\F(L_{0})=S$.
	\end{Remark}

\medskip
\begin{Lemma}\label{cover}
	Let $k,n$ be two positive integers with $n>k\geq 3$ and $C\subseteq \{x_i\mid i\in [k-1]\}$. Then
	$C$ is a minimal vertex cover of $\mathcal{M}(L_k)$ if and only if $C$ is a minimal vertex cover of $\mathcal{M}(L_n)$.
\end{Lemma}
\begin{proof}
	It is enough to show that $C$ is a  vertex cover of $\mathcal{M}(L_k)$ if and only if $C$ is a  vertex cover of $\mathcal{M}(L_n)$.
	
	$(\Rightarrow)$ Let $F=\{x_{i_1},\ldots,x_{i_j},x_{i_{j+1}},\ldots,x_{i_\ell}\}\in \text{Facets}\,(\mathcal{M}(L_n))$
	with $1\leq i_1< \cdots<i_j\leq k< i_{j+1} < \cdots< x_{i_\ell}\leq n$. then $i_{j+1}-i_j\in \{2,3\}$ by Remark \ref{remark1}. It follows that
	$i_j\in \{k-2,k-1,k\}$. Take $F'=\{x_{i_t}\mid 1\leq t\leq j\}$.
	
	Claim: $F'\cap C\neq \emptyset$. Thus $F\cap C\neq \emptyset$, as desired.
	
	Indeed, if $i_j\in \{k-1,k\}$, then $F'\in \text{Facets}\,(\mathcal{M}(L_k))$. This implies that $F'\cap C\neq \emptyset$
	because $C$ is a  vertex cover of $\mathcal{M}(L_k)$. If $i_j=k-2$, then $F'\cup \{x_k\}\in \text{Facets}\,(\mathcal{M}(L_k))$. Thus $F'\cap C\neq \emptyset$
	since $x_k\notin C$ and  $C$ is a  vertex cover of $\mathcal{M}(L_k)$.
	
	$(\Leftarrow)$ Let $F\in \text{Facets}\,(\mathcal{M}(L_k))$, the $F$ can always be extended to a facet $F'$ in $\mathcal{M}(L_n)$. Thus $F'\cap C\neq \emptyset$
	since $C$ is a  vertex cover of $\mathcal{M}(L_n)$. By hypothesis that $C\subseteq \{x_1,\ldots,x_{k-1}\}$, we obtain that if  $x_j\in C\cap F'$ then $x_j\in F\cap C$, as desired.
\end{proof}

\medskip
By direct calculation, we obtain
\begin{Lemma}\label{decompose1}
	Let  $L_n$  be a line graph with edge set  $\{x_i\mid i\in [n]\}$ where $1\leq n\leq 5$.
	Then $P$ is a minimal prime over $\F(L_n)$ if and only if $P$ satisfies the following conditions:
	\begin{itemize}
		\item[(1)] If $n=1$, then $P=(x_1)$;
		\item[(2)] If $n=2$, then $P=(x_1,x_2)$;
		\item[(3)] If $n=3$, then $P=(x_1,x_2)$ or $P=(x_2,x_3)$;
		\item[(4)] If $n=4$, then $P=(x_1,x_2)$ or $P=(x_3,x_4)$ or $P=(x_1,x_4)$;
		\item[(5)] If $n=5$, then $P=(x_1,x_2)$ or $P=(x_4,x_5)$ or $P=(x_2,x_3,x_4)$.
	\end{itemize}
\end{Lemma}
Let $\Omega_n$ be a collection of all minimal vertex covers of $\mathcal{M}(L_n)$ with  $n\in [5]$.
Now, we assume that $n\geq 6$ is  an integer. We write $n$ as  $n=3p+d$, where $p\geq 2$ and $0\leq d\leq 2$. Set
\begin{eqnarray*}
	A_1&=&\{x_1,x_2\}, \ \ \ A'_1=\{x_{n-1},x_n\};\\
	A_i&=&\{x_{i},x_{i+1},x_{i+2}\},\ \text{ for\ }2\leq i\leq n-3;\\
	B_j&=&\{x_{3\ell-2}\mid 1\leq \ell\leq j+1\}\cup\{x_{3j+2}\},\ \text{ for\ }1\leq j\leq p-1;\\
	B'_j&=&\{x_{n-3(\ell-1)}\mid 1\leq \ell\leq j+1\}\cup\{x_{n-3j-1}\},\ \text{ for\ }1\leq j\leq p-1;\\
	D&=&\{x_{3j+1}\mid 0\leq j\leq p\}, \text{\ if\ } n=3p+1.
\end{eqnarray*}
$$
\Omega_{n}=\left
\{\begin{array}{l@{\  \ }l}
	\{A_1,A'_1,A_2,A_3,B_1,B'_1\}& \text{if\ } n=6,\\
	\{A_1,A'_1,A_2,A_3,A_4,B_1,B'_1,D\}& \text{if\ } n=7,
\end{array}\right.
$$
When $n\geq 8$.
If   $n=3p$ or $3p+1$, we set
$C_{k\ell}=\{x_{k},x_{k+3\ell+2}\}\cup\{x_{k+3j-2}\mid 1\leq j\leq \ell+1\}$ for $k\in [n-3(\ell+1)]\setminus [1],\ell\in [p-2]$.
If $n=3p+2$, we set $C_{k\ell}=\{x_{k},x_{k+3\ell+2}\}\cup\{x_{k+3j-2}\mid 1\leq j\leq \ell+1\}$, for $k\in [n-3(\ell+1)]\setminus [1], \ell\in [p-1]$.
Put
$$\Omega_n=\left
\{\begin{array}{l@{\  \ }l}
	\left\{A_i,A'_1,B_j,B'_j,C_{k\ell}\mid \begin{subarray}{l} \text{for\ }i\in [n-3],\, j\in [p-1],\\
		k\in [n-3(\ell+1)]\setminus [1]\text{\ and\ }\ell\in [p-2]\end{subarray}\right\}&\ \text{if  $n=3p$},\\
	\left\{A_i,A'_1,B_j,B'_j,C_{k\ell},D\mid \begin{subarray}{l} \text{for\ }i\in [n-3],\, j\in [p-1],\\
		k\in [n-3(\ell+1)]\setminus [1]\text{\ and\ }\ell\in [p-2] \end{subarray}\right\}&\ \text{if  $n=3p+1$},\\
	\left\{A_i,A'_1,B_j,B'_j,C_{k\ell}\mid \begin{subarray}{l} \text{for\ }i\in [n-3],\, j\in [p-1],\\
		k\in [n-3(\ell+1)]\setminus [1]\text{\ and\ }\ell\in [p-1] \end{subarray}\right\}&\ \text{if   $n=3p+2$}.
\end{array}\right.$$

\medskip
By Lemmas \ref{prime}, \ref{cover} and the definition of $\Omega_n$, we have
\begin{Remark}\label{remark3}
	Let $s,n$ be two positive integers with $n>s\geq 3$. Then
	$$\{C\in \Omega_n\mid C\subseteq \{x_i\,|\, i\in [s-1]\}\}\subseteq \Omega_{s}.$$
\end{Remark}

\begin{Theorem}\label{decompose2}
	Let $n$ be a positive integer,  $L_n$ a line graph with edge set  $\{x_i\mid i\in [n]\}$.
	Then $P$ is a minimal prime over $\F(L_n)$ if and only if  $P$ is  generated by   $C$ where  $C\in \Omega_n$.
\end{Theorem}
\begin{proof}
	Since $\F(L_n)$ is a squarefree monomial ideal, we obtain by \cite[Corollary 1.3.6]{HH1} that
	$$\F(L_n)=\bigcap\limits_{P\in \text{Min}\,(\F(L_n))}P$$
	where $\text{Min}\,(\F(L_n))$ is the set of minimal prime ideals over $\F(L_n)$.
	Hence it is enough to show that $\Omega_n$ is a collection of all minimal vertex covers of $\mathcal{M}(L_n)$ by Lemma \ref{prime}.
	We apply induction on $n$. Cases $1\leq n\leq 5$ follow from Lemma \ref{decompose1}. Cases $n=6,7$ are obtained  by direct calculation.  Suppose that $n\geq 8$ and that the statement holds for any integer  $6\leq t\leq n-1$.
	Now we prove that the statement holds for $n$.

	First, we  will show that  $C$ is a minimal vertex cover of $\mathcal{M}(L_n)$ for any $C\in \Omega_n$. Let $T_n$ be the set of minimal vertex covers of $\mathcal{M}(L_n)$.
	
	We distinguish into the following four cases:
	
	(1) If $C=A_1$ or $C=A'_1$, then it is obvious  $C\in T_n$ from a fact
	$$\mathcal{G}(\F(L_n))=\mathcal{G}(x_1\F(L'_{n-2}))\cup\mathcal{G}(x_2\F(L'_{n-3}))=\mathcal{G}(x_n\F(L_{n-2}))\cup\mathcal{G}(x_{n-1}\F(L_{n-3})).$$
	
	(2) If $C\cap A'_1=\emptyset$. Then one can easily see that
	$$
	C\notin\left
	\{\begin{array}{l@{\  \ }l}
		\left\{A'_1,A_{n-3},B_{p-1},B'_j, C_{n-3(\ell+1),\ell}\mid \begin{subarray}{c} \text{for\ }\ j\in [p-1],\\
			\ell\in [p-2] \end{subarray}\right\}& \text{if\ } n=3p,\\
		\left\{A'_1,A_{n-3},D,B'_j, C_{n-3(\ell+1),\ell}\mid \begin{subarray}{c} \text{for\ }\ j\in [p-1],\\
			\ell\in [p-2] \end{subarray}\right\}& \text{if\ } n=3p+1,\\
		\left\{A'_1,A_{n-3},B'_j, C_{n-3(\ell+1),\ell}\mid  \text{for\ }j,\ell\in [p-1]\right\}& \text{if\ } n=3p+2.
	\end{array}\right.
	$$
	By the definition of $\Omega_{n-1}$, it is easy to see that
	$C\in \Omega_{n-1}$. Thus $C\in T_{n-1}$ by   inductive hypothesis. Since $C\subseteq  \{x_i\mid i\in [n-2]\}$, we get  $C\in T_{n}$ by Lemma \ref{cover}.
	
	(3) If $C\cap A_1=\emptyset$.  Let $y_i=x_{n+1-i}$ for any $i\in [n]$, then $L_n$ is not only a line graph with edges $y_1,\ldots,y_n$, but also satisfies
	$C\cap \{y_{n-1},y_n\}=\emptyset$. It follows that $C\in T_{n}$ from the proof of (2) above.
	
	(4) If $C\cap A_1\neq \emptyset$ and  $C\cap A'_1\neq \emptyset$. One can obtain  that
	$$
	C\in\left
	\{\begin{array}{l@{\  \ }l}
		\left\{B_{p-1},B'_{p-1}\right\}& \text{if\ } n=3p,\\
		\left\{D\right\}& \text{if\ } n=3p+1,\\
		\left\{ C_{2,p-1}\right\}& \text{if\ } n=3p+2.
	\end{array}\right.
	$$
	We divide into  the following three  cases:
	
	(i) When  $n=3p$. If $C=B_{p-1}$,  then
	$$C=\{x_{3\ell-2}\mid \ell\in [p]\}\cup\{x_{3p-1}\}=\{x_{3\ell-2}\mid \ell\in [p]\}\cup\{x_{n-1}\}.$$
	Let $C'=\{x_{3\ell-2}\mid \ell\in [p]\}$, then  $C=C'\sqcup \{x_{n-1}\}$ and $C'\in \Omega_{n-2}$ by the definition of $\Omega_{n-2}$. Hence  $C'\in T_{n-2}$ by inductive hypothesis.
	For any  $F\in \text{Facets}\,(\mathcal{M}(L_n))$,  it can be written as
	$ F_1\cup \{x_{n}\}$ or $F_2\cup \{x_{n-1}\}$,
	where  $F_1\in \text{Facets}\,(\mathcal{M}(L_{n-2}))$ and $F_2\in \text{Facets}\,(\mathcal{M}(L_{n-3}))$.
	It follows that $C\in T_n$.
	
	If $C=B'_{p-1}$. Let $y_i=x_{n+1-i}$ for each  $i$, then $L_n$ is also a line graph with   edges $y_1,\ldots,y_n$. In this case, one has
	$$C=\{y_{3\ell-2}\mid \ell\in [p]\}\cup\{y_{3p-1}\}=\{y_{3\ell-2}\mid \ell\in [p]\}\cup\{y_{n-1}\}.$$
	It follows that $C\in T_n$ from the above proof.
	
	(ii) If $n=3p+1$, then  $C=D=\{x_{3j-2}\mid j\in [p+1]\}$. Let $C''=\{x_{3j-2}\mid j\in [p]\}$, then  $C=C''\sqcup \{x_{n}\}$ and $C''\in \Omega_{n-3}$ by the definition of $\Omega_{n-3}$.
	By inductive hypothesis, we obtain  $C''\in T_{n-3}$. It follows that $C\in T_n$ from the proof of  (i).
	
	(iii) If   $n=3p+2$, then  $C=C_{2,p-1}=\{x_{2},x_{3p+1}\}\cup\{x_{3j}\mid j\in [p]\}=\{x_{2},x_{n-1}\}\cup\{x_{3j}\mid j\in [p]\}$.
	Let $C'''=\{x_{3j}\mid j\in [p]\}$. According to (ii), we obtain that $C'''$ is a minimal vertex cover of $\mathcal{M}(L''_{n-4})$, where $L''_{n-4}$ is a line graph with edge set $\{x_i\mid 3\leq i\leq n-2\}$.
	Note that
	\begin{eqnarray*}
		\mathcal{G}(\F(L_n))&=&\mathcal{G}(x_1\F(L'_{n-2}))\cup\mathcal{G}(x_2\F(L'_{n-3}))\\
		&=&\mathcal{G}(x_1x_n\F(L''_{n-4}))\cup\mathcal{G}(x_1x_{n-1}\F(L''_{n-5}))\cup\mathcal{G}(x_2\F(L'_{n-3})),
	\end{eqnarray*}
	thus for any $F\in \text{Facets}\,(\mathcal{M}(L_n))$, it can be written in one of the following three different forms:
	$$(a)\ \ F_1\cup \{x_1,x_{n}\},\ (b)\ \ F_2\cup \{x_1,x_{n-1}\},\ (c)\ \ F_3\cup \{x_2\},$$
	where $F_1\in \text{Facets}\,(\mathcal{M}(L''_{n-4}))$,
	$F_2\in \text{Facets}\,(\mathcal{M}(L''_{n-5}))$, $F_3\in \text{Facets}\,(\mathcal{M}(L'_{n-3}))$, $L''_{n-4}$, $L''_{n-5}$  and $L'_{n-3}$ are line graphs with $\{x_i\mid 3\leq i\leq n-2\}$, $\{x_i\mid 3\leq i\leq n-3\}$, and $\{x_i\mid 4\leq i\leq n\}$ as edge sets respectively. Hence $C\in T_n$.
	
	\medskip
	Next,  we will show that any vertex cover $C$ of $\mathcal{M}(L_n)$ must contain
	some element in $\Omega_{n}$. Thus $\Omega_{n}$ is  exactly a collection of all minimal vertex covers of $\mathcal{M}(L_n)$, as desired.
	
	We divide into the following two cases:

	(1) If  $A'_1\cap C=\emptyset$, then $C\subseteq \{x_i\mid i\in [n-2]\}$. It follows that $C$ is a vertex cover of $\mathcal{M}(L_{n-1})$ by Lemma \ref{cover}. By induction hypothesis, we obtain
	$C'\subseteq C$ for some $C'\in \Omega_{n-1}$. According to Remark \ref{remark3}, one has $C'\in \Omega_{n}$ because of $C'\subseteq \{x_i\mid i\in [n-2]\}$,  as
	desired.

	(2) If  $A'_1\cap C\neq\emptyset$, then $C$ satisfies one of the following three conditions:
	$$(i)\ A'_1\cap C=A'_1,\ \ (ii) \ A'_1\cap C=\{x_{n}\},\ \ (iii)\ A'_1\cap C=\{x_{n-1}\}. $$
	
	$(i)$ If $A'_1\cap C=A'_1$, obviously $C\supseteq A'_1$.
	
	$(ii)$ If $A'_1\cap C=\{x_{n}\}$. Set
	$$q=\max\{\ell\mid\, x_{n-3s}\in C \ \text{for any\ } 0\leq s\leq \ell\}.$$
	It follows that $\{x_{n-3s}\mid 0\leq s\leq q\}\subseteq C$.
	
	We consider the following three cases:
	
	(a) If $n-3q=1$, then   $q=p$ from  the notation of $n$. In this case,  $C\supseteq D$.
	
	(b) If $n-3q\geq 2$ and  there exists some $x_{n-3s-1}\in C$ with $s\in [q]$, then
	$$C\supseteq\{x_{n-3j}\mid 0\leq j\leq s\}\cup \{x_{n-3s-1}\}.$$
	Hence
	$$C\supseteq\left\{\begin{array}{l@{\  \ }l}
		A_1& \text{if\ } n=3q+2,\\
		B'_s& \text{otherwise}.
	\end{array}\right.
	$$
	
	(c)  If $n-3q\geq 2$ and  $x_{n-3s-1}\notin C$ for any $s\in [q]$, then $n-3q\geq 6$.
	Indeed,  we choose $F=\{x_{n+2-3j}\mid\, j\in [q+1]\}$, then $F\cap C=\emptyset$.
	If $n-3q=2$ or $3$, then $F\in \text{Facets}\,(\mathcal{M}(L_n))$,  contradicting with the assumption that $C$ is a vertex cover of $\mathcal{M}(L_n)$.
	If $n-3q=4$ or $5$, then $F\cup\{x_{n-3q-3}\}\in \text{Facets}\,(\mathcal{M}(L_n))$, and $x_{n-3q-3}\notin C$ by the choice of $q$, a contradiction again.
	
	Let  $C'=C\setminus \{x_{i}\mid n-3q-2\leq i\leq n\}$. Claim: $C'$ is a vertex cover of $\mathcal{M}(L_{n-3q-3})$.
	It follows that $C'\supseteq \Gamma$ for some $\Gamma\in \Omega_{n-3q-3}$ by  induction hypothesis. Note a fact that $x_{n-3q-3}\notin C$, which implies $\Gamma\subseteq \{x_i|i\in [n-3q-4]\}$. Hence $\Gamma\in \Omega_{n}$ by Lemma \ref{cover}.

	{\em The proof of Claim:}
	If the assertion does not hold, then  there exists some $F'\in \text{Facets}\,(\mathcal{M}(L_{n-3q-3}))$  such that $F'\cap C'=\emptyset$. According to the definition of
	matching complex, we get $\max\{i\mid x_i\in F'\}=n-3q-3, \text{or}\ n-3q-4$.
	Choose $\Lambda=F'\cup F=F'\cup\{x_{n+2-3j}\mid\, j\in [q+1]\}$. Then  $\Lambda\in \text{Facets}\,(\mathcal{M}(L_n))$, and $\Lambda\cap C=\emptyset$  by the supposition that $x_{n-3s-1}\notin C$ for any $s\in [q]$ and $F'\cap C'=\emptyset$.
	It contradicts with  the assumption that $C$ is a vertex cover of $\mathcal{M}(L_n)$.
	
	$(iii)$ If $A'_1\cap C=\{x_{n-1}\}$, then $C'$ is a vertex cover of $\mathcal{M}(L_{n-2})$  by
	Remark \ref{remark1}, where $C'=C\setminus\{x_{n-1}\}$. Hence  $C'\supseteq \Gamma$ for some $\Gamma\in \Omega_{n-2}$ by  induction hypothesis.
	We divide into the following two cases:
	
	(a)If $x_{n-2}\notin \Gamma$, then, by  Lemma \ref{cover}, we have  $\Gamma\in \Omega_{n}$ because of  $\Gamma\subseteq \{x_i\mid\, i\in [n-3]\}$, as desired.
	
	(b) If $x_{n-2}\in \Gamma$, then, from the definition of $\Omega_{n-2}$ and a fact $\Gamma\in \Omega_{n-2}$, it can be obtained that
	$$
	\Gamma\in\left
	\{\begin{array}{l@{\  \ }l}
		\left\{A''_1, B''_{j}, D'\mid\, j\in[p'-1]\right\}& \text{if\ } n=3p\\
		\left\{A''_1,B''_{j}\mid\, j\in[p'-1]\right\}& \text{if\ } n=3p+1\ \text{or}\ 3p+2,
	\end{array}\right.
	$$
	where $A''_1=\{x_{n-3}, x_{n-2}\}$,  $B''_j=\{x_{n+1-3\ell}\mid  \ell\in [j+1]\}\cup\{x_{n-3(j+1)}\}$ for any $j\in [p'-1]$ with
	$p'=\left
	\{\begin{array}{l@{\  \ }l}
		p-1& \text{if\ } n=3p\ \text{or}\ 3p+1\\
		p& \text{if\ } n=3p+2
	\end{array}\right.$,
	and $D'=\{x_{3\ell-2}\mid \ell\in  [p-1]\}$ when $n=3p$.
	
	If $\Gamma=A''_1$, then $\Gamma\cup \{x_{n-1}\}=A_{n-3}\in \Omega_{n}$; If $\Gamma=B''_{j}$ for some $j\in [p'-1]$, then $\Gamma\cup \{x_{n-1}\}=C_{n-3j-3,\,j}\in\Omega_{n}$; If $n=3p$ and $\Gamma=D'$, thus $\Gamma\cup \{x_{n-1}\}=B_{p-1}\in \Omega_{n}$.
	
	In all  cases it follows that $\Gamma\cup \{x_{n-1}\}\in \Omega_{n}$, as desired.
	This finishes the proof.
\end{proof}

\medskip
Given an ideal $I\subset S$, we set
$$\text{bight}\,(I)=\sup\{\height\,(P)\mid  P \text{ is a   minimal prime  ideal  of } S \text{ over }  I\}.$$

\medskip
As a consequence of the above theorem, we have
\begin{Corollary}\label{ht}
	Let $n$ be a positive integer,  $L_n$ a line graph with edge set $\{x_i\mid i\in [n]\}$. Then
	\begin{enumerate}
		\item[(a)] $\height\,(\F(L_2))=1$ and $\height\,(\F(L_n))=2$ for any  $n\geq3$;
		\item[(b)] $\text{{\em bight}}\, (\F(L_n))=\left\{\begin{array}{ll}
			p+1\ \ &\text{if}\ \ n=3p\text{\ or\ } n=3p+1,\\
			p+2\ \ &\text{if}\ \ n=3p+2.
		\end{array}\right.$
	\end{enumerate}
\end{Corollary}

	\medskip
	
	\section{Projective dimension and regularity  of  facet ideals of  the matching complexes of   line graphs}
	
	In this section, we will provide some formulas for  the  projective dimension and regularity of  the facet ideal $\F(L_{n})$ of the matching complex of a line graph $L_n$ with n edges.

	\medskip
	\begin{Lemma}\label{line1}
		Let $n\geq 5$ be an integer, $L_n$ a line graph with edge set $\{x_i\mid i\in [n]\}$. Let $\F(L_n)$ be the facet ideal
		of the  simplicial complex $\mathcal{M}(L_n)$. Then
		$$\F(L_n)=J_n+K_n\ \text{and \ }J_n\cap K_n=x_1x_2P_{n},$$
		where $J_n=x_1\F(L'_{n-2})$, $K_n=x_2\F(L'_{n-3})$,  $P_{n}=x_4\F(L'_{n-5})+x_3x_5\F(L'_{n-6})$ and
		we stipulate $\F(L'_{-1})=\F(L'_{0})=S$.
	\end{Lemma}
	\begin{proof}
		It's obvious that $\F(L_n)=J_n+K_n$  by Remark \ref{remark2}. By direct calculation, we  obtain
		$J_5\cap K_5=x_1x_2(x_3x_5,x_4)=x_1x_2P_5$ and $J_6\cap K_6=x_1x_2(x_3x_5,x_4x_6)=x_1x_2P_6$, as desired.

		Assume that  $n\geq 7$. By repeated using Remark \ref{remark2}, for all $1\leq i\leq 4$, we have  $$\F(L'_{n-i})=x_{i+1}\F(L'_{n-i-2})+x_{i+2}\F(L'_{n-i-3}),$$
		where $\{x_{i+1}\}\cap \text{supp}\,(\F(L'_{n-i-3}))=\{x_{i+2}\}\cap \text{supp}\,(\F(L'_{n-i-2}))=\emptyset$.
		One has
		$$J_n\cap K_n=x_1x_2(\F(L'_{n-2})\cap \F(L'_{n-3})).$$
		By comparing the expressions of $J_n\cap K_n$, we get $P_n=\F(L'_{n-2})\cap \F(L'_{n-3})$. It follows that
		\begin{eqnarray*}
			P_n&=&(x_3\F(L'_{n-4})+x_4\F(L'_{n-5}))\cap (x_4\F(L'_{n-5})+x_5\F(L'_{n-6}))\\
			&=&x_4\F(L'_{n-5})+(x_3\F(L'_{n-4}))\cap (x_5\F(L'_{n-6}))+(x_4\F(L'_{n-5}))\cap (x_5\F(L'_{n-6}))\\
			&=&x_4\F(L'_{n-5})+(x_3\F(L'_{n-4}))\cap (x_5\F(L'_{n-6}))\\
			&=&x_4\F(L'_{n-5})+[x_3(x_5\F(L'_{n-6})+x_6\F(L'_{n-7})]\cap(x_5\F(L'_{n-6}))\\
			&=&x_4\F(L'_{n-5})+x_3x_5\F(L'_{n-6})+(x_3x_6\F(L'_{n-7}))\cap(x_5\F(L'_{n-6}))\\
			&=&x_4\F(L'_{n-5})+x_3x_5\F(L'_{n-6})+x_3x_5x_6(\F(L'_{n-7})\cap\F(L'_{n-6}))\\
			&=&x_4\F(L'_{n-5})+x_3x_5\F(L'_{n-6}).
		\end{eqnarray*}
		where the penultimate equality holds because of $\{x_3,x_6\}\cap \text{supp}\,(\F(L'_{n-6}))=\{x_5\}\cap \text{supp}\,(\F(L'_{n-7}))=\emptyset$.
	\end{proof}
	
	\begin{Theorem}\label{line2}
		Let $n$ be a positive integer, $L_n$ a line graph with edge set $\{x_i\mid i\in [n]\}$. Let $\F(L_n)$ be the facet ideal
		of the  simplicial complex $\mathcal{M}(L_n)$.
		Then
		\[
		\pd\,(\F(L_n))\leq\left\{\begin{array}{ll}
			p\ \ &\text{if}\ n=3p \text{\ or\ } n=3p+1,\\
			p+1\ \ &\text{if}\ n=3p+2,
		\end{array}\right.
		\]
		and
		\[
		\reg\,(\F(L_n))=\left\{\begin{array}{ll}
			1\ \ &\text{if}\  \  n=1,\\
			2p\ \ &\text{if}\ \ n=3p, \text{\ or\ } n=3p+1 \text{\ with \ }p>1,\\
			2p+1\ \ &\text{if}\ \ n=3p+2.
		\end{array}\right.
		\]
		Moreover, if $n\geq 5$, then
		$$\pd(P_n)\leq  \left\{\begin{array}{ll}
			p-1\ \ &\text{if}\ n=3p \text{\ or\ } n=3p+1,\\
			p\ \ &\text{if}\ \ n=3p+2,
		\end{array}\right.$$
		and
		$$\reg(P_n)=\left\{\begin{array}{ll}
			2p-1\ \ &\text{if}\ n=3p \text{\ or\ } n=3p+1,\\
			2p\ \ &\text{if}\ n=3p+2,
		\end{array}\right.
		$$
		where  $P_n$ as defined  in Lemma \ref{line1}.
	\end{Theorem}
	\begin{proof}
		Cases $1\leq n\leq 3$ are clear by Lemma \ref{sum} and case $n=4$ follows from \cite[Theroem 3.7]{Zhu}. Assume that  $n\geq 5$.
		Consider the following short exact sequence
		$$0\longrightarrow J_n\cap K_n \longrightarrow J_n\oplus K_n\longrightarrow \F(L_n)\longrightarrow 0,\eqno(\dag)$$
		where $J_n$ and $K_n$,  as defined in Lemma \ref{line1}, satisfy $J_n\cap K_n=x_1x_2P_{n}$.
		Applying Lemma \ref{exact} (2), (4) and Lemma \ref{supp} (2), (3) to the short exact sequence above, we obtain
		\begin{eqnarray*}
			\pd\,(\F(L_n))&\leq& \max\,\{\pd\,(J_n\cap K_n)+1, \pd\,(J_n), \pd\,(K_n)\}\\
			&=&\max\,\{\pd\,(P_{n})+1, \pd\,(\F(L'_{n-2})), \pd\,(\F(L'_{n-3}))\},
		\end{eqnarray*}
		and
		\begin{eqnarray*}
			\reg\,(\F(L_n))&\leq& \max\,\{\reg\,(J_n\cap K_n)-1, \reg\,(J_n), \reg\,(K_n)\}\\
			&=&\max\,\{2+\reg\,(P_{n})-1, 1+\reg\,(\F(L'_{n-2})), 1+\reg\,(\F(L'_{n-3}))\},\\
			&=&1+\max\,\{\reg\,(P_{n}), \reg\,(\F(L'_{n-2})), \reg\,(\F(L'_{n-3}))\},
		\end{eqnarray*}
		where equality holds if $\reg\,(P_{n})\neq\max\,\{ \reg\,(\F(L'_{n-2})), \reg\,(\F(L'_{n-3}))\}-1$.
		
		\medskip
		Since $L'_{n-k}$ is also a line graph with  edge set  $\{x_i\mid i\in [n]\setminus [k]\}$ for any positive integers  $n>k$, we have
		$\pd\,(\F(L'_{n-k}))=\pd\,(\F(L_{n-k}))$ and $\reg\,(\F(L'_{n-k}))=\reg\,(\F(L_{n-k}))$. Thus, from the formulas for the projection dimension and regularity above, we obtain
		$$
		\pd\,(\F(L_n))\leq \max\,\{\pd\,(P_{n})+1, \pd\,(\F(L_{n-2})), \pd\,(\F(L_{n-3}))\},   \hspace{2.5cm}(1)
		$$
		and
		$$
		\reg\,(\F(L_n))\leq1+\max\,\{\reg\,(P_{n}), \reg\,(\F(L_{n-2})), \reg\,(\F(L_{n-3}))\},   \hspace{2.2cm}(2)
		$$
		where equality holds if $\reg\,(P_{n})\neq\max\,\{ \reg\,(\F(L_{n-2})), \reg\,(\F(L_{n-3}))\}-1$.
		
		\medskip
		We will prove these assertions by induction on $n$.
		
		By direct calculation, one has  $ P_5=(x_4,x_3x_5)$, $P_6=(x_4x_6,x_3x_5)$, $\F(L_{2})=(x_1,x_2)$, $ \F(L_{3})=(x_1x_3,x_2)$  and $ \F(L_{4})=(x_1x_3,x_1x_4,x_2x_4)$.
		It follows that
		$$\pd\,(P_5)=\pd\,(P_6)=\pd\,(\F(L_{2}))=\pd\,(\F(L_{3}))=\pd\,(\F(L_{4}))=1,$$
		$$\reg\,(P_5)=2, \ \reg\,(P_6)=3,\  \reg\,(\F(L_{2}))=1,\ \reg\,(\F(L_{2}))=\reg\,(\F(L_{4}))=2.$$
		Hence, according to the formulas $(1)$ and $(2)$, we have
		$$
		\pd\,(\F(L_5))\leq 2,\ \ \pd\,(\F(L_6))\leq 2,\ \ \reg\,(\F(L_5))=3,\ \ \reg\,(\F(L_6))=4.
		$$
		This proves the assertion for $n=5$ and $n=6$.
		
		If $n=7$, then $P_7=(x_3x_5x_7,x_4x_6,x_4x_7)$. It follows that
		$(P_7:x_4)=(x_6,x_7)$ and $(P_7,x_4)=(x_3x_5x_7,x_4)$.
		One has  $\pd\,((P_7:x_4))=\pd\,((P_7,x_4))=1$, $\reg\,((P_7:x_4))=1$ and $\reg\,((P_7,x_4))=3$ by Lemma \ref{sum} (1) and (3). Applying  Lemma \ref{exact}  to the following  short exact sequence
		$$0\longrightarrow (P_7:x_4)(-1)\stackrel{ \cdot x_{4}} \longrightarrow P_7\longrightarrow (P_7,x_{4})\longrightarrow 0,$$
		we obtain $\pd\,(P_7)=1$ and $\reg\,(P_7)=3$.
		By inductive hypothesis,  we have
		$$\pd\,(\F(L_{4}))\leq 1, \ \pd\,(\F(L_{5}))\leq2, \ \reg\,(\F(L_{4}))=2, \ \reg\,(\F(L_{5}))=3.$$
		It follows that
		$\pd\,(\F(L_7))\leq 2$ and $\reg\,(\F(L_7))=4$ by formulas $(1)$ and $(2)$, as desired.
		
		Assume that $n\geq 8$. Set $M_n=x_4\F(L'_{n-5})$ and $N_n=x_3x_5\F(L'_{n-6})$.  One has
		$$\pd\,(M_n)=\pd\,(\F(L_{n-5}))\leq p-1,\ \  \pd\,(N_n)=\pd\,(\F(L_{n-6}))\leq p-2,\eqno(3)$$
		\begin{eqnarray*}
			\hspace{1cm}\reg\,(M_n)&=& 1+\reg\,(\F(L_{n-5}))=1+\left\{\begin{array}{ll}
				2p-3\ \ &\text{if}\ n=3p \text{\ or\ }n=3p+1\\
				2(p-1)\ \ &\text{if}\  n=3p+2
			\end{array}\right.\\
			&=&\left\{\begin{array}{ll}
				2p-3\ \ &\text{if}\ n=3p\\
				2p-2\ \ &\text{if}\ n=3p+1\\
				2p-1\ \ &\text{if}\  n=3p+2,
			\end{array}\right.    \hspace{5.8cm}(4)\\
			\reg\,(N_n)&=& 2+\reg\,(\F(L'_{n-6}))=2+\left\{\begin{array}{ll}
				2(p-2) &\text{if}\ n=3p \text{\ or\ } n=3p+1\\
				2(p-2)+1&\text{if}\  n=3p+2
			\end{array}\right.\\
			&=&\left\{\begin{array}{ll}
				2p-2\ \ &\text{if}\ n=3p \text{\ or\ }n=3p+1\\
				2p-1\ \ &\text{if}\  n=3p+2.
			\end{array}\right.    \hspace{4.1cm}(5)
		\end{eqnarray*}
		
		Since $\{x_{4}\}\cap \text{supp}\,(\F(L'_{n-6}))=\{x_{3},x_{5}\}\cap \text{supp}\,(\F(L'_{n-5}))=\emptyset$,
		one has
		$$P_n=M_n+N_n\ \ \text{and\ \ }M_n\cap N_n=x_3x_4x_5(\F(L'_{n-5})\cap \F(L'_{n-6})).
		$$
		Combining  Lemma \ref{supp} (2) and (3), the expressions of $M_n\cap N_n$ and the inductive hypothesis, we obtain
		\begin{eqnarray*}
			\pd\,(M_n\cap N_n)&=&\pd\,(\F(L'_{n-5})\cap \F(L'_{n-6}))=\pd\,(\F(L'_{(n-3)-2})\cap \F(L'_{(n-3)-3}))\\
			&=&\pd\,(P_{n-3})\leq \left\{\begin{array}{ll}
				p-2\ \ &\text{if}\ n=3p \text{\ or\ } n=3p+1\\
				p-1\ \ &\text{if}\ \ n=3p+2   \hspace{4.0cm}(6)
			\end{array}\right.
		\end{eqnarray*}
		\begin{eqnarray*}
			\reg\,(M_n\cap N_n)&=&3+\reg\,(\F(L'_{n-5})\cap \F(L'_{n-6}))\\
			&=&3+\reg\,(\F(L'_{(n-3)-2})\cap \F(L'_{(n-3)-3}))\\
			&=&3+\reg\,(P_{n-3})=3+\left\{\begin{array}{ll}
				2(p-1)-1\ \ &\text{if}\ n=3p \text{\ or\ } n=3p+1\\
				2(p-1)\ \ &\text{if}\ n=3p+2
			\end{array}\right.\\
			&=&\left\{\begin{array}{ll}
				2p\ \ &\text{if}\ n=3p \text{\ or\ } n=3p+1\\
				2p+1\ \ &\text{if}\ n=3p+2 \hspace{6.0cm}(7)
			\end{array}\right.
		\end{eqnarray*}
		Applying  Lemma \ref{exact} (2), (4) and formulas $(3) \sim (7)$ to the following short exact sequence
		$$0\longrightarrow M_n\cap N_n \longrightarrow  M_n\oplus N_n\longrightarrow P_n\longrightarrow 0, \eqno(\dag\dag)$$
		we obtain
		\begin{eqnarray*}
			\pd\,(P_n)&\leq &\max\,\{\pd\,(M_n\cap N_n)+1, \pd\,(M_n), \pd\,(N_n)\}\\
			&\leq & \left\{\begin{array}{ll}
				p-1\ \ &\text{if}\ n=3p \text{\ or\ } n=3p+1\\
				p\ \ &\text{if}\ \ n=3p+2
			\end{array}\right. \hspace{5.2cm}(8)\\
			\reg\,(P_n)&=&\max\,\{\reg\,(M_n\cap N_n)-1, \reg\,(M_n), \reg\,(N_n)\} \\
			&=&\left\{\begin{array}{ll}
				2p-1\ \ &\text{if}\ n=3p \text{\ or\ } n=3p+1\\
				2p\ \ &\text{if}\ n=3p+2
			\end{array}\right. \hspace{5.0cm}(9)
		\end{eqnarray*}
		Combining  formulas  (1), (2), (7), (8)  and  inductive hypothesis, one has
		\[
		\pd\,(\F(L_n))\leq\left\{\begin{array}{ll}
			p\ \ &\text{if}\ n=3p \text{\ or\ } n=3p+1\\
			p+1\ \ &\text{if}\ n=3p+2,
		\end{array}\right.
		\]
		and
		\[
		\reg\,(\F(L_n))=\left\{\begin{array}{ll}
			2p\ \ &\text{if}\ \ n=3p, \text{\ or\ } n=3p+1\\
			2p+1\ \ &\text{if}\ \ n=3p+2.
		\end{array}\right.
		\]
		This finishes the proof.
	\end{proof}

	\begin{Lemma} \label{ht1}{\em (\cite[Corollary 3.33]{MV})} Let  $I\subset S$  be a squarefree monomial ideal, then
		$$\text{{\em bight}}\, (I)\leq \pd\,(I)+1.$$
	\end{Lemma}
	
	By Corollary \ref{ht}, Theorem \ref{line2}  and the  lemma above, we have
	\begin{Corollary}\label{line3}
		Let $n$ be a positive integer, $L_n$  defined as in Lemma 3.1. Then
		$$
		\pd\,(\F(L_n))= \left\{\begin{array}{ll}
			p\ \ &\text{if}\ \ n=3p\text{\ \ or\ \ }n=3p+1\\
			p+1\ \ &\text{if}\ \ n=3p+2.
		\end{array}\right.
		$$
	\end{Corollary}
	
	An immediate consequence of the above corollary is the following corollary
	\begin{Corollary}\label{line3}
		Let $n$ be a positive integer, $L_n$  defined as in Lemma 3.1. Then
		$$
		\depth\,(\F(L_n))= \left\{\begin{array}{ll}
			2p\ \ &\text{if}\ \ n=3p\\
			2p+1\ \ &\text{if}\ \ n=3p+1\text{\ \ or\ \ }n=3p+2.
		\end{array}\right.
		$$
	\end{Corollary}
	\begin{proof}It follows from Auslander-Buchsbaum formula (see  \cite[Theorem 1.3.3]{BH}).
	\end{proof}
	
	\medskip
	
	\medskip
	\hspace{-6mm} {\bf Acknowledgments}
	
	\vspace{3mm}
	\hspace{-6mm}  This research is supported by the National Natural Science Foundation of China (No.11271275) and  by foundation of the Priority Academic Program Development of Jiangsu Higher Education Institutions.

	\medskip

\end{document}